\documentclass[a4paper,12pt,oneside,dvipsnames,reqno]{amsart} 

\usepackage{amsmath,amsfonts,amsthm,amssymb}

\usepackage{xcolor}
\usepackage{graphicx}
\usepackage{cite}

\usepackage{bbm}
\usepackage{tikz-cd}

\usepackage{hyperref}
\hypersetup{colorlinks=true, citecolor=PineGreen, linkcolor=RoyalBlue, linktoc=page,urlcolor=RoyalBlue}

\usepackage{geometry}

\theoremstyle{plain}
\newtheorem{theorem}{Theorem}[section] 
\newtheorem{lemmy}[theorem]{Lemma}

\newtheorem{cor}[theorem]{Corollary}

\theoremstyle{remark}
\newtheorem{rem}[theorem]{Remark}

\theoremstyle{definition}

\newcommand{\R}{\mathbb{R}}

\newcommand{\C}{\mathbb{C}}

\newcommand{\Hb}{\mathbb{H}}

\newcommand{\Z}{\mathbb{Z}}

\newcommand{\Q}{\mathbb{Q}}

\newcommand{\GL}{\operatorname{GL}}

\newcommand{\SL}{\operatorname{SL}}

\newcommand{\Tr}{\operatorname{Tr}}

\newcommand{\N}{\mathbb{N}}

\title[Fourier coefficients of Siegel modular forms]{New bounds for fundamental Fourier coefficients of Siegel modular forms}

\author{Edgar Assing}

\date{\today}

\email{assing@math.uni-bonn.de}

\thanks{The author is supported by the Germany Excellence Strategy grant EXC-2047/1-390685813 and also partially funded by the Deutsche Forschungsgemeinschaft (DFG, German Research Foundation) – Project-ID 491392403 – TRR 358.}

\begin{document}

\begin{abstract}
We prove new bounds for the Fourier coefficients of Jacobi forms using a method of Iwaniec. In view of the Fourier-Jacobi expansion of degree two Siegel modular forms, we can use these to obtain strong bounds on fundamental Fourier coefficients of Siegel modular forms.
\end{abstract}

\maketitle

\section{Introduction}

Let $F\in S^{(2)}_k(\textrm{Sp}_4(\Z))$ be a cuspidal (holomorphic) Siegel modular form of genus $2$ and weight $k$. We have the Fourier expansion
\begin{equation}
	F(Z) = \sum_{T>0} a_F(T) e(\Tr(TZ)), \nonumber
\end{equation}
where the sum is taken over symmetric matrices with half integral entries (and integral diagonal). It is an interesting problem to study the growth of the coefficients $a_F(T)$ with respect to suitable parameters that capture the \textit{complexity} of $T$. The trivial bound, which reads
\begin{equation}
	a_F(T)\ll_F \det(T)^{\frac{k}{2}},\nonumber
\end{equation}
follows from a classical Hecke-style argument. If $F$ is not a Saito-Kurokawa lift or $T$ is fundamental, then it is believed that the bound
\begin{equation}
	a_F(T)\ll_{F,\epsilon} \det(T)^{\frac{k}{2}-\frac{3}{4}+\epsilon} \nonumber
\end{equation}
holds. This is (a version of) \cite[Conjecture~IV]{RS} by Resnikoff and Salda\~na. Progress towards this conjecture was made by Kitaoka in \cite{Ki}, where it was shown that
\begin{equation}
	a_F(T)\ll_{F,\epsilon} \det(T)^{\frac{k}{2}-\frac{1}{4}+\epsilon}. \label{eq:Kitaoka}
\end{equation}
This improvement is achieved by using a Petersson-type formula for Siegel modular forms, which is nowadays called Kitaoka's formula.

On the other hand, Raghavan and Weissauer established
\begin{equation}
	a_F(T)\ll_{F,\epsilon} \min(T)^{\frac{1}{2}}\det(T)^{\frac{k}{2}-\frac{1}{4}-\frac{3}{38}+\epsilon}, \nonumber
\end{equation}
where $\min(T)$ denotes the least positive integer represented by $T$. Note that this bound improves \eqref{eq:Kitaoka} only when $\min(T)$ is relatively small. 

Finally it was shown by Kohnen in \cite{Ko, Ko2} that
\begin{equation}
	a_F(T)\ll_{F,\epsilon} \min(T)^{\frac{5}{18}}\det(T)^{\frac{k}{2}-\frac{1}{2}+\epsilon}. \label{eq:Kohnen_for_siegel}
\end{equation}
Using the bound $\min(T)\ll \det(T)^{\frac{1}{2}}$ coming from reduction theory one obtains
\begin{equation}
	a_F(T)\ll_{F,\epsilon} \det(T)^{\frac{k}{2}-\frac{13}{36}+\epsilon}. \nonumber
\end{equation}
This bound has been generalized to congruence subgroups and more general weights in \cite{Ho}. However, to the best of our knowledge, there is no general unconditional improvement of \eqref{eq:Kohnen_for_siegel} available yet.

In this paper we improve Kohnen's result if $\min(T)$ is not too large. Our main result is the following.

\begin{theorem}\label{th:Siegel_bound}
Let $k>2$ be even and let $F\in S^{(2)}_k(\textrm{Sp}_4(\Z))$. Then, for fundamental $T$ (i.e. $-4\det(T)$ is a fundamental discriminant), we have
\begin{equation}
	a_F(T) \ll_{F,\epsilon} \left(1+\frac{\det(T)}{\min(T)^{\frac{25}{7}}}\right)^{-\frac{3}{144}}\cdot \det(T)^{\frac{k}{2}-\frac{1}{2}+\epsilon}\min(T)^{\frac{5}{18}}.\nonumber
\end{equation}
\end{theorem}

This improves upon \eqref{eq:Kohnen_for_siegel} as soon as $\min(T)\ll \det(T)^{\frac{7}{25}}$. We also obtain the following amusing corollary, which features a slight improvement on \cite[Theorem~1 and Theorem~2]{Ko3}.

\begin{cor}\label{cor}
Let $k>2$ be even, $F\in S^{(2)}_k(\textrm{Sp}_4(\Z))$ and let $D<0$ be a fundamental discriminant. Then, in any genus of $\textrm{Cl}_D/\textrm{Cl}_D^2$, there exists $T$ with $D=-4\det(T)$ such that
\begin{equation}
 	a_F(T) \ll_{F,\epsilon} \vert D\vert^{\frac{k}{2}-\frac{1}{2}+\frac{131}{2016}+\epsilon}.\nonumber
\end{equation} 
If we assume the generalized Lindel\"of hypothesis (GLH) for $L(s,\chi_{\Delta})$ for all fundamental discriminants $\Delta$, then we even find $T$ as above with
\begin{equation}
	a_F(T) \ll_{F,\epsilon} \vert D\vert^{\frac{k}{2}-\frac{1}{2}-\frac{1}{144}+\epsilon}.\nonumber
\end{equation} 
\end{cor}
\begin{proof}
We follow the arguments from \cite[Section~6]{Ko3}. Unconditionally, we find $T$ in the genus that represents a positive integer $m$ with $(m,D)=1$ and $m\ll \vert D\vert^{\frac{1}{4}+\epsilon}$. See \cite{H}. The bound in the statement follows after specializing Theorem~\ref{th:Siegel_bound} to $\min(T)\ll \det(T)^{\frac{1}{4}+\epsilon}$.
	
On the other hand, assuming GLH we can argue as in \cite[p.11]{Ko3} to see that any genus contains a form $T$ that represents a prime $p$ with $(D,p)=1$ and $p\ll_{\epsilon} \vert D\vert^{\epsilon}$. The desired bound follows from Theorem~\ref{th:Siegel_bound} after specializing to $\min(T) \ll \det(T)^{\epsilon}$.
\end{proof}

\begin{rem}
Recently progress in our understanding of the GGP conjecture and its relation to Fourier coefficients of Siegel modular forms allowed for some progress towards the Resnikoff-Salda\~na Conjecture under GRH. A relevant result for fundamental Fourier coefficients of Yoshida lifts can be found in \cite[Corollary~4]{BB}. Fundamental Fourier coefficients of general forms are considered \cite{JLS} and these results are extended to arbitrary Fourier coefficients in \cite{CMS}.
\end{rem}

\begin{rem}
It is an interesting problem to study these questions for Siegel modular forms of $\textrm{Sp}_{2n}(\Z)$ (or congruence subgroups thereof). There are results available in this direction, but the exponents become worse for growing $n$. We refer to \cite{BK,Bre, Br} for example.
\end{rem}

\subsection{Reduction to Jacobi forms}

The key input into establishing Theorem~\ref{th:Siegel_bound} is a new bound for Fourier coefficients of Jacobi forms. Before we state the result let us explain the reduction process.

Recall that we are assuming that $F\in S^{(2)}_k(\textrm{Sp}_4(\Z))$ with $k>2$ even. It is well known that for $A\in \GL_2(\Z)$ we have
\begin{equation}
	a_F(T) = a_F(A^tTA) \text{ for all }A\in \GL_2(\Z).\nonumber
\end{equation}
In particular, using reduction theory, we can reduce the case where
\begin{equation}
	T=\left(\begin{matrix} n_0 & \frac{r_0}{2} \\ \frac{r_0}{2} & m_0\end{matrix}\right), \nonumber
\end{equation}
where $0<D=r_0^2-4m_0n_0=-4\det(T)$ and $m_0=\min(T)>0$. By reduction theory we have the bound $m_0\ll \vert D\vert^{\frac{1}{2}}$.

Next, we rewrite the Fourier expansion of $F$ in terms of Fourier-Jacobi coefficients
\begin{equation}
	F(Z) = \sum_{m\geq 1} \phi_m(\tau,z)e(m\tau') \text{ where } Z=\left(\begin{matrix} \tau & z \\ z & \tau'\end{matrix}\right)\in \Hb^{(2)}.\nonumber
\end{equation}
Here $\phi_m\in J_{k,m}^{\textrm{cusp}}$ are certain Jacobi forms of weight $k$ and index $m$ (for the full Jacobi group $\Gamma^J$). See Section~\ref{section_jac} for relevant definitions concerning Jacobi forms. At this point we only recall that the Jacobi forms $\phi_m$ have a Fourier expansion of the form
\begin{equation}
		\phi_{m}(\tau,z) = \sum_{\substack{n_1,r_1\in \Z \\ r^2<4mn}} c_{\phi_m}(n,r)e(n\tau+rz),\, \tau\in \Hb^{(1)} \text{ and }z\in \C.\nonumber
\end{equation}
Clearly we must have
\begin{equation}
	\vert a_F(T)\vert = \frac{\vert c_{\phi_{m_0}}(n_0,r_0)\vert}{\Vert \phi_{m_0}\Vert}\cdot \Vert \phi_{m_0}\Vert \label{eq:reduc1}
\end{equation}
for $T$ as above. We have normalized the Fourier coefficients naturally by the Petersson norm of $\phi_{m_0}$.

At this point we recall the following theorem.
\begin{theorem}\label{th:Pet_Fjac}
Let $F\in S^{(2)}_k(\textrm{Sp}_4(\Z))$ with Fourier-Jacobi coefficients $\{\phi_m\}_{m\in \N}$. Then we have
\begin{equation}
	\Vert\phi_m\Vert \ll_{F,\epsilon} m^{\frac{k}{2}-\frac{2}{9}+\epsilon}.
\end{equation}
\end{theorem}

\begin{rem}
In \cite[(10)]{KS} it is conjectured that 
\begin{equation}
	\Vert \phi_m\Vert \ll m^{\frac{k-1}{2}}.\nonumber
\end{equation}
However, the bound given in Theorem~\ref{th:Pet_Fjac} is still the best unconditional bound available.
\end{rem}

If we insert the result from Theorem~\ref{th:Pet_Fjac} into \eqref{eq:reduc1} we obtain
\begin{equation}
	\vert a_F(T)\vert = \frac{\vert c_{\phi_{m_0}}(n_0,r_0)\vert}{\Vert \phi_{m_0}\Vert}\cdot m_0^{\frac{k}{2}-\frac{2}{9}+\epsilon} \label{eq:reduc2}
\end{equation}
It all boils down to obtaining suitable bounds for $\frac{\vert c_{\phi_{m_0}}(n_0,r_0)\vert}{\Vert \phi_{m_0}\Vert}$. Note that, since we do not have any useful information on $\phi_{m_0}$ these can only depend on the weight $k$. In \cite{Ko2, Ko} the following breakthrough result is established.

\begin{theorem}[Kohnen 1993]\label{Kohnen}
Let $f\in J_{k,m}^{\textrm{cusp}}$ with $k>2$ and suppose that $D=r^2-4mn<0$. Then
\begin{equation}
	c_f(n,r) \ll_{k,\epsilon} \vert D\vert^{\epsilon}\cdot \left(1+\frac{m}{\vert D\vert^{\frac{1}{2}}}\right)^{\frac{1}{2}}\cdot \left( \frac{\vert D\vert}{m}\right)^{k/2-1/2}\cdot \Vert f\Vert. \nonumber
\end{equation}
\end{theorem}

Inserting this in \eqref{eq:reduc2} and using $\vert D\vert =4\det(T)$ and $m_0 =\min(T)\ll \det(T)^{\frac{1}{2}}$ produces the bound \eqref{eq:Kohnen_for_siegel}. Our main result, namely Theorem~\ref{th:Siegel_bound}, follows directly from the following bound for Fourier coefficients of Jacobi forms.

\begin{theorem}\label{th:bound_jacobi}
Let $f\in J_{k,m}^{\textrm{cusp}}$ with $k>2$ even and suppose that $D=r^2-4mn<0$ is a fundamental discriminant. We have
\begin{equation}
	c_f(n,r) \ll_{k,\epsilon}\vert D\vert^{\epsilon}\cdot \left(1+\frac{\vert D\vert}{m^{\frac{25}{7}}}\right)^{-\frac{3}{144}}\cdot \left(\frac{\vert D\vert}{m}\right)^{k/2-1/2}\cdot \Vert f\Vert. \nonumber
\end{equation}
\end{theorem}

\begin{rem}
Note that this improves upon Kohnen's result stated in Theorem~\ref{Kohnen} for $m\leq \vert D\vert^{\frac{7}{25}}$. We have not attempted to fully optimize the exponents and expect that the method can be squeezed to give improved exponents in certain ranges. An interesting question that remains unanswered at this point is whether an improvement in the exponent can be achieved as soon as $m\ll \vert D\vert^{\frac{1}{2}-\delta}$ for $\delta>0$ arbitrarily small. The barrier $m\leq \vert D\vert^{\frac{7}{25}}$ arising in the theorem can probably be improved slightly. However, getting close to $\vert D\vert^{\frac{1}{2}}$ might be difficult with the techniques used here.
\end{rem}

\subsection{The method}

We will now explain the idea behind the proof of Theorem~\ref{th:bound_jacobi}. In absence of any other obvious tools we start from a Petersson type formula for Jacobi forms, see Theorem~\ref{th:petersson} below for details. For $f\in J_{k,m}^{\textrm{cusp}}$ with $k>2$ we obtain a bound of the form
\begin{equation}
	\frac{\vert c_f(n,r)\vert^2}{\Vert f\Vert^2} \ll_k (\sqrt{m}+\vert \mathcal{S}_1\vert)\cdot \left(\frac{\vert D\vert}{m}\right)^{k-\frac{3}{2}}, \nonumber
\end{equation}
where 
\begin{equation}
	\mathcal{S}_N = \sum_{\pm}\sum_{\substack{c\geq 1,\\ N\mid c}}^{\infty}c^{-\frac{3}{2}}H_{m,c}^{\pm}(n,r)e\left(\pm \frac{r^2}{2mc}\right)J_{k-\frac{3}{2}}\left(\frac{\pi\vert D\vert}{mc}\right). \nonumber
\end{equation}
The Kloosterman like sums $H_{m,c}^{\pm}(n,r)$ are defined in \eqref{eq:def_KS} below. In view of the Weil type bound
\begin{equation}
	H^{\pm}_{m,c}(n,r) \ll_{\epsilon} c^{1+\epsilon}(c,D), \label{eq:basic_bound_KS}
\end{equation} 
given in \cite[(3)]{Ko2}, we can estimate the sum $\mathcal{S}_N$ trivially.\footnote{By trivially we mean that potential cancellation between the terms of the $c$-sum is not exploited.} We obtain $$\mathcal{S}_1 \ll_{k,\epsilon} \vert D\vert^{\frac{1}{2}+\epsilon}m^{-\frac{1}{2}}$$ and thus
\begin{equation}
	\frac{\vert c_f(n,r)\vert^2}{\Vert f\Vert^2} \ll_{k,\epsilon} (\sqrt{m}+\frac{\vert D\vert^{\frac{1}{2}+\epsilon}}{m^{\frac{1}{2}}})\cdot \left(\frac{\vert D\vert}{m}\right)^{k-\frac{3}{2}}. \label{eq:Kohnen_type}
\end{equation}
This is precisely the bound from Theorem~\ref{Kohnen}.

Note that, if $m\asymp \vert D\vert^{\frac{1}{2}}$, then 
\begin{equation}
	\sqrt{m} \asymp \frac{\vert D\vert^{\frac{1}{2}}}{m^{\frac{1}{2}}} \asymp \vert D\vert^{\frac{1}{4}} \nonumber
\end{equation}
and we have no chance of improving  on \eqref{eq:Kohnen_type}. In practice one hopes that there is a lot of cancellation between the terms in the sum $\mathcal{S}_1$. If one can estimate the sum $\mathcal{S}_1$ non-trivially, then one can hope for improved bounds as soon $m\ll \vert D\vert^{\frac{1}{2}-\delta}$. 

In order to treat $\mathcal{S}_1$ non-trivial we will exploit the fact that the sums $H_{m,c}^{\pm}(n,r)$ are essentially Sali\'e sums. In particular, we can explicitly evaluate them in many cases. See Lemma~\ref{lm:eval_KS} below. With this evaluation at hand the sum $\mathcal{S}_1$ closely resembles the sum that appears on the geometric side of the Petersson formula for classical modular forms of half integral weight. This allows us to follow the ideas introduced in \cite{Iw_half} to extract further cancellation. See also \cite[Section~5.3]{Iw} and \cite[Chapter~4]{Sa} for detailed discussions.

Very briefly the idea of Iwaniec is as follows. Using the fact that a Jacobi form in $S_{k,m}^{\textrm{cusp}}$ is automatically a cuspidal Jacobi form of level $N$ we can artificially introduce an average of $\mathcal{S}_{\ast}$:
\begin{equation}
	\frac{\vert c_f(n,r)\vert^2}{\Vert f\Vert^2} \ll_k \left(\sqrt{m}P+\left\vert\sum_{p\asymp P}\log(p) \mathcal{S}_p\right\vert\right)\cdot \left(\frac{\vert D\vert}{m}\right)^{k-\frac{3}{2}}, \nonumber
\end{equation}
See \eqref{eq:starting_point} below for the precise statement. The average over $p\asymp P$  leads to certain bilinear forms which are approachable using classical tools. Ultimately we rely on two distinct sources of cancellation. In certain ranges we encounter sums featuring the Jacobi symbols $\left(\frac{D}{\cdot}\right)$. Here the assumption that $D$ is a fundamental discriminant is crucial, since it allows us to apply the Polya-Vinogradov inequality. On the other hand one exploits equidistribution of solutions to $v^2\equiv 1\text{ mod }q$. In deriving the relevant estimates we closely follow the ideas of Iwaniec from \cite{Iw_half} as presented in \cite{Iw}.

The paper is organized as follows. In Section~\ref{section_kloo} we recall basic properties of the Kloosterman sums $H_{m,c}^{\pm}(n,r)$. This is followed by Section~\ref{section_jac} summarizing the (relevant part of the) basic theory of Jacobi forms. The heart of the paper is contained in Section~\ref{subsec}, which contains the main analytic arguments going in the proof of Theorem~\ref{th:bound_jacobi}.

\subsection*{Acknowledgements}

I would like to thank V. Blomer for useful comments on a first draft of this paper.

\section{Kloosterman sums}\label{section_kloo}

In this section we study the complete exponential sum
\begin{equation}
	H_{m,c}^{\pm}(n,r) = {\sum_{\rho \text{ mod }c}}^{\ast}e\left(\frac{n}{c}(\overline{\rho}+\rho)\right)\sum_{\lambda\text{ mod }c}e\left(\frac{m\overline{\rho}\lambda^2+r(\overline{\rho}\pm 1)\lambda}{c}\right).\label{eq:def_KS}
\end{equation}
These are sums of Kloosterman-type and turn out to be closely related to classical Sali\'e sums. For us they naturally show up on the geometric of the Petersson formula for Jacobi forms given in Theorem~\ref{th:petersson} below.

We start by providing an explicit evaluation whenever the modulus $c$ is co-prime to $2mc$. This is certainly well known in the Jacobi form community, but since the existence of this explicit evaluation is crucial to our argument we provide the details. 

\begin{lemmy}\label{lm:eval_KS}
Let $(c,2mD)=1$. Then we have
\begin{equation}
	H_{m,c}^{\pm}(n,r) = \epsilon_c^2\cdot c\cdot \left(\frac{-D}{c}\right)\sum_{v^2\equiv 1\text{ mod }c}e\left(\frac{\overline{2m}Dv}{c}\mp \frac{\overline{2m}r^2}{c}\right).\nonumber
\end{equation}	
Here $\epsilon_c=1$ if $c\equiv 1\text{ mod }4$ and $\epsilon_c=i$ if $c\equiv 3\mod 4$. 
\end{lemmy}
\begin{proof}
We first observe that 
\begin{equation}
	m\overline{\rho}\lambda^2 + r(\overline{\rho}\pm 1)\lambda \equiv m\overline{\rho}(\lambda+r\overline{2m}(1\pm \rho))^2 -r^2\overline{4m\rho}-r^2\overline{4m}\rho \mp \overline{2m}r^2 \text{ mod }c.\nonumber
\end{equation}
Using this and a change of variables we can compute the $\lambda$-sum as follows:
\begin{align}
	\sum_{\lambda\text{ mod }c}e\left(\frac{m\overline{\rho}\lambda^2+r(\overline{\rho}\pm 1)\lambda}{c}\right) &= e\left(-\frac{\overline{4m}r^2}{c}(\overline{\rho}+\rho)\right)e\left(\mp \frac{\overline{2m}r^2}{c}\right)\sum_{\lambda \text{ mod }c}e\left(\frac{m\overline{\rho}}{c}\lambda^2\right) \nonumber\\
	& = \epsilon_ce\left(-\frac{\overline{4m}r^2}{c}(\overline{\rho}+\rho)\right)e\left(\mp \frac{\overline{2m}r^2}{c}\right)\left(\frac{m\overline{\rho}}{c}\right)c^{\frac{1}{2}}.\nonumber
\end{align}
Here we have used the evaluation of the Gau\ss\  sum given for example in \cite[p. 52, Exercise 4]{IK}. All together we obtain
\begin{equation}
	H_{m,c}^{\pm}(n,r) = \epsilon_c\cdot c^{\frac{1}{2}}e\left(\mp \frac{\overline{2m}r^2}{c}\right){\sum_{\rho \text{ mod }c}}^{\ast}\left(\frac{m\rho}{c}\right)\left(\frac{-\overline{4m}D}{c}(\overline{\rho}+\rho)\right).\nonumber
\end{equation}
The remaining $\rho$-sum is precisely a Sali\'e sum and it can be computed on the nose. See for example \cite[Lemma~12.4]{IK}. We obtain
\begin{equation}
	H_{m,c}^{\pm}(n,r) = \epsilon_c^2\cdot c\cdot \left(\frac{-D}{c}\right)\sum_{v^2\equiv (-D\overline{4m})^2\text{ mod }c}e\left(\frac{2v}{c}\mp \frac{\overline{2m}r^2}{c}\right)
\end{equation}
After a change of variables in the $v$-sum we are done.
\end{proof}

To handle the remaining moduli $c$ we will need some further properties of the sums $H_{m,c}^{\pm}(n,r)$. First, we recall that they exhibit typical factorization behavior inherited from the Chinese Remainder Theorem. More precisely, if $c=c_1c_2$ with $(c_1,c_2)=1$, then we can factor the Kloosterman sum as
\begin{equation}
	H_{m,c}^{\pm}(n,r) = H_{mc_1,c_2}^{\pm}(n\overline{c_1},r)\cdot H_{mc_2,c_1}^{\pm}(n\overline{c_2},r). \label{eq:factorization}
\end{equation}
See for example \cite[p. 173]{Ko2}. Further, we have already mentioned the Weil type bound \eqref{eq:basic_bound_KS}. This bound was proven for fundamental discriminants in \cite{Ko} and in \cite{Ko2} in general. Actually something slightly stronger is true. We will use the following.

\begin{lemmy}\label{lm:imp_Kl_bound}
Let $D=r^2-4mn$. We have
\begin{equation}
	\vert H_{m,c}^{\pm}(n,r)\vert \leq \tau(c)c\cdot (c,m,r)^{\frac{1}{2}}(c,D')^{\frac{1}{2}}, \nonumber
\end{equation}
where $D'=\frac{D}{(m,r)}$. If $D$ is a fundamental discriminant, then $(m,r)$ and $D'$ are co-prime and our estimate simplifies to 
\begin{equation}
	\vert H_{m,c}^{\pm}(n,r)\vert \leq \tau(c)c\cdot (c,D)^{\frac{1}{2}}. \nonumber
\end{equation}
\end{lemmy}
While the proof of this is essentially contained in \cite{Ko2} and also in \cite[Lemma~3.2]{Br}, we will still give full details for convenience of the reader.
\begin{proof}
After factoring $H_{m,c}^{\pm}(n,r)$ it is sufficient to consider $c=p^s$. Further, by Lemma~\ref{lm:eval_KS} we have
\begin{equation}
	\vert H_{m,p^s}(n,r)\vert  \leq  2p^s \text{ for }p\nmid 2mD.\nonumber
\end{equation}
Thus, we still need to consider $p\mid 2mD$. We will treat the case $p\neq 2$ below. The case $p=2$ can be handled similarly using slightly modified formulae for (quadratic) Gau\ss\ sums. We omit the details.  
	
Suppose $p\mid mD$ and $p\neq 2$. We define $e,f\in \Z_{\geq 0}$ by $p^e=(m,p^s)$ and $p^f=(r,p^s)$. We consider several cases. We write $m=m'p^e$ and $r=r'p^f$.
	
First, we treat the exceptional case when $e=f=s$. Note that in this case $p^s \mid D$. In this case the $\lambda$-sum in the definition of $H^{\pm}_{m,p^s}$ is trivial and we have
\begin{equation}
	H_{m,c}^{\pm} = p^s\cdot {\sum_{\rho \text{ mod }p^s}}^{\ast} e\left(\frac{n}{p^s}(\overline{\rho}+\rho)\right).\nonumber
\end{equation}
From Weil's bound for the classical Kloosterman sum we obtain
\begin{equation}
	\vert H_{m,p^s}^{\pm}(n,r)\vert \leq  2p^{\frac{3}{2}s}(p^s,n)^{\frac{1}{2}} \leq 2p^s(p^s,m,r)^{\frac{1}{2}}(p^s,D')^{\frac{1}{2}}.\nonumber
\end{equation}
	
Second, take $0\leq f< e\leq s$. We compute the $\lambda$-sum as follows
\begin{align}
	&\sum_{\lambda\text{ mod }p^s}e\left(\frac{m\overline{\rho}\lambda^2+r(\overline{\rho}\pm 1)\lambda}{p^s}\right) \nonumber\\
	&\qquad = \sum_{\lambda \text{ mod }p^{s-e}}e\left(\frac{m'\overline{\rho}\lambda^2}{p^{s-e}}+\frac{r'(\overline{\rho}\pm 1)\lambda}{p^{s-f}}\right)\sum_{x\text{ mod }p^e} e\left(\frac{r'(\overline{\rho}\pm 1)x}{p^{e-f}}\right) \nonumber\\
	&\qquad = \delta_{\rho\equiv \mp 1\text{ mod }p^{e-f}}p^e\sum_{\lambda \text{ mod }p^{s-e}}e\left(\frac{m'\overline{\rho}\lambda^2}{p^{s-e}}+\frac{r'(\overline{\rho}\pm 1)\lambda}{p^{s-f}}\right)\nonumber \\
	&\qquad = \delta_{\rho\equiv \mp 1\text{ mod }p^{e-f}}\epsilon_{p^{s-e}}\left(\frac{m'\overline{\rho}}{p^{s-e}}\right)\cdot e\left(\frac{-\overline{4m'}r^2(\overline{\rho}+\rho)p^{-e}\mp \overline{2m'}r^2p^{-e}}{p^s}\right) \cdot p^{\frac{s+e}{2}}.\nonumber
\end{align}
We are thus left with
\begin{equation}
	H_{m,p^s}^{\pm}(n,r) = \epsilon_{p^{s-e}}p^{\frac{s+e}{2}}\cdot e\left(\mp \frac{\overline{2m'}r^2p^{-e}}{p^s}\right)\cdot  \sum_{\substack{\rho \text{ mod }p^s \\ \rho \equiv \mp 1 \text{ mod }p^{e-f}}} \left(\frac{m'\rho}{p^{s-e}}\right) e\left(\frac{n-\overline{4m'}r^2p^{-e}}{p^s}(\overline{\rho}+\rho)\right).\nonumber
\end{equation}
At this point we replace $\rho$ by $\mp 1 + \rho p^{e-f}$ and observe that
\begin{equation}
	\overline{\mp 1 + \rho p^{e-f}} \equiv  \mp 1\pm \rho p^{e-f} \mp \rho^2p^{2e-2f}+\ldots \text{ mod }p^s.\nonumber
\end{equation}
Recall that in the case at hand $D'= (r')^2p^f-4m'np^{e-f}$. Inserting this above gives
\begin{multline}
	H_{m,p^s}^{\pm}(n,r) = \epsilon_{p^{s-e}}p^{\frac{s+e}{2}}\cdot e\left(\mp \frac{2n}{p^s}\right)\left(\frac{\mp m'}{p}\right)^{s-e} \\\cdot  \sum_{\rho \text{ mod }p^{s+f-e}} e\left(-\overline{4m'}\frac{D'}{p^s}[(1\pm 1)\rho \mp \rho^2p^{e-f} \pm \ldots]\right).\nonumber
\end{multline}
The remaining sum is $\ll (s+1)p^{\frac{s+f-e}{2}}(p^s,D')^{\frac{1}{2}}$. This leads to the bound
\begin{equation}
	\vert H_{m,p^s}^{\pm}(n,r)\vert \leq (s+1)p^{s+\frac{f}{2}}(p^s,D')^{\frac{1}{2}}. \nonumber
\end{equation}
Because $p^f=(p^s,m,r)$ we are done.

Next we look at the case $0\leq e\leq f\leq s$, where $D'=Dp^{-e} = (r')^2p^{2f-e}-4m'n$. Again we start by observing that the $\lambda$-sum is
\begin{multline}
	\sum_{\lambda\text{ mod }p^s}e\left(\frac{m\overline{\rho}\lambda^2+r(\overline{\rho}\pm 1)\lambda}{p^s}\right) = p^e\sum_{\lambda \text{ mod }p^{s-e}}e\left(\frac{m'\overline{\rho}\lambda^2+r'p^{f-e}(\overline{\rho}\pm 1)\lambda}{p^{s-e}}\right) \\ 
	= p^{\frac{s+e}{2}}\epsilon_{p^{s-e}}\left(\frac{m'\overline{\rho}}{p^{s-e}}\right)e\left(\mp \frac{\overline{2m'}(r')^2p^{2f-2e}}{p^{s-e}}-\frac{\overline{4m'}(r')^2p^{2f-2e}}{p^{s-e}}(\overline{\rho}+\rho)\right).\nonumber
\end{multline}
Altogether we obtain
\begin{equation}
	H_{m,p^s}^{\pm}(n,r) = \epsilon_{p^{s-e}}p^{\frac{s+e}{2}}\cdot e\left(\mp \frac{\overline{2m'}(r')^2p^{2f-e}}{p^{s}}\right)\cdot  \sum_{\rho \text{ mod }p^s } \left(\frac{m'\rho}{p^{s-e}}\right) e\left(\frac{-\overline{4m'}D'}{p^s}(\overline{\rho}+\rho)\right).\nonumber
\end{equation}
The remaining sum, which is very similar to the usual Kloosterman sum or Sali\'e sum, can be estimated directly and exhibits square root cancellation. We obtain
\begin{equation}
	\vert H_{m,p^s}^{\pm}(n,r)\vert \leq 2p^{s+\frac{e}{2}}(p^s,D')^{\frac{1}{2}}.\nonumber
\end{equation}
We are done because $p^e = (p^s,m,r)$. This was the last case to consider and the proof is complete.
\end{proof}

We conclude this section by providing a useful estimate concerning an (incomplete) exponential sum that we will need later on. While this has nothing to do with the sums $H_{m,c}^{\pm}(n,r)$ studied so far, the proof will use the completion method which in the case at hand produces classical Kloosterman sums. This justifies the placement of this lemma in the current section.

\begin{lemmy}\label{lm:completeion}
Let $A,L\geq 0$ be parameters and let $a,c,s,t\in \N$. We have
\begin{equation}
	\sum_{\substack{A< x<A+L \\ x\equiv s\text{ mod }t \\  (x,c)=1}}e\left(\frac{a\overline{x}}{c}\right) \ll_{\epsilon} (cL)^{\epsilon}\left(\frac{L}{c}(a,c)+c^{\frac{1}{2}}\right).\nonumber
\end{equation}
\end{lemmy} 
This is a slightly modified version of \cite[Lemma~8]{Iw_half}. While the modifications are rather straightforward we work a little bit in order to remove a factor $(a,c)^{\frac{1}{2}}$ in the second term of the estimate. This will be important for us later on and it is the reason why we include a proof of this otherwise relatively basic result.
\begin{proof}
The proof uses the completion method as described in \cite[Section~12.2]{IK} for example. First, at the cost of an $O(1)$-error, which is absorbed by the second term of our bound, we can assume that $A\in \Z$. Without loss of generality we assume that $c>0$.

We define 
\begin{equation}
	F(x) = \delta_{(x,c)=1}e\left(\frac{a\overline{x}}{c}\right).\nonumber
\end{equation}	
This is a $c$-periodic function and we define its Fourier transform by 
\begin{equation}
	\widehat{F}(y) = \sum_{x\text{ mod }c}F(x)e\left(-\frac{yx}{c}\right). \nonumber
\end{equation}
After applying Fourier inversion we get
\begin{align}
	\sum_{\substack{A< x<A+L \\ x\equiv s\text{ mod }t \\  (x,c)=1}}e\left(\frac{a\overline{x}}{c}\right) &= 	\sum_{\substack{A< x<A+L \\ x\equiv s\text{ mod }t}}\frac{1}{c}\sum_{y\text{ mod }c}\widehat{F}(y)e\left(\frac{yx}{c}\right) \nonumber\\
	&= \frac{\widehat{F}(0)}{c}\sum_{\substack{A< x<A+L \\ x\equiv s\text{ mod }t}}1 + \sum_{\substack{-\frac{c}{2}<y\leq \frac{c}{2},\\ y\neq 0}} \frac{\widehat{F}(y)}{c}\cdot\underbrace{\sum_{\substack{A< x<A+L \\ x\equiv s\text{ mod }t }}e\left(\frac{yx}{c}\right)}_{=\lambda\left(\frac{y}{c}\right)}.\nonumber
\end{align}
A standard argument using the geometric series shows that
\begin{equation}
	\left\vert \lambda\left(\frac{y}{c}\right)\right\vert \leq \frac{2}{\vert 1-e(\frac{y}{c})\vert} \ll \frac{c}{\vert y\vert}.\nonumber
\end{equation}
for $0<\vert y\vert \leq \frac{c}{2}$. We arrive at
\begin{equation}
	\sum_{\substack{A< x<A+L \\ x\equiv s\text{ mod }t \\  (x,c)=1}}e\left(\frac{a\overline{x}}{c}\right) \ll \frac{L}{c}\vert \widehat{F}(0)\vert + \sum_{\substack{-\frac{c}{2}<y\leq \frac{c}{2},\\ y\neq 0}} \frac{1}{\vert y\vert}\cdot\vert\widehat{F}(y)\vert.
\end{equation}
The remaining task is to analyse $\widehat{F}$. In our case this simply turns out to be a Kloosterman sum.

Note that $\widehat{F}(0)$ is a Ramanujan sum and by \cite[(3.3)]{IK} we have
\begin{equation}
	\vert\widehat{F}(0)\vert =\mu\left(\frac{c}{(c,a)}\right)^2\frac{\varphi(c)}{\varphi(\frac{c}{(c,a)})}  \leq (a,c).\nonumber
\end{equation}
This gives us the desired main term. Next, from \cite[Corollary~11.12]{IK} for example, we recall that
\begin{equation}
	\vert \widehat{F}(y)\vert = \vert S(-y,a;c)\vert \leq \tau(c)(y,a,c)^{\frac{1}{2}}\cdot c^{\frac{1}{2}} \nonumber
\end{equation}
by the Weil bound. Finally, we have Selberg's identity
\begin{equation}
	S(-y,a;c) = \sum_{l\mid (y,a,c)}l\cdot S(-\frac{ay}{l^2},1,\frac{c}{l}).\nonumber
\end{equation}
So far we have seen that 
\begin{equation}
	\sum_{\substack{A< x<A+L \\ x\equiv s\text{ mod }t \\  (x,dc)=1}}e\left(\frac{a\overline{x}}{c}\right) \ll \frac{L}{c}(a,c) + \sum_{\substack{-\frac{c}{2}<y\leq \frac{c}{2},\\ y\neq 0}} \frac{1}{\vert y\vert}\cdot\vert S(y,-a;c)\vert. \nonumber
\end{equation}
To handle the second sum we use Selberg's identity and estimate
\begin{align}
	\sum_{1\leq y\leq X}\frac{1}{y}\vert S(y,-a;c)\vert &\ll \sum_{t\mid (a,c)} \sum_{1\leq y\leq X/l}\frac{1}{y}\vert S(\frac{-ay}{l},1;\frac{c}{l})\vert\nonumber \\
	&\ll \sum_{l\mid (a,c)} \sum_{1\leq y\leq X/l}\frac{1}{y} \tau(c/l)\cdot (c/l)^{\frac{1}{2}} \ll_{\epsilon} X^{\epsilon}c^{\frac{1}{2}+\epsilon}.\nonumber
\end{align}
Plugging this in above produces the desired bound.
\end{proof}

\section{Jacobi forms}\label{section_jac}

We start by recalling some basic definitions of Jacobi forms. We took the material from \cite{Ib}, but the classical reference is \cite{EZ}.

For a ring $R\subseteq \R$ we write $\GL_2^+(R)$ for the group of $2\times 2$ matrices over $R$ with determinant in $R^{\times}\cap \R_{>0}$.  Further let $H(R) = R^2\times R$ denote the Heisenberg group. We define the Jacobi group over $R$ as 
\begin{equation}
	G^J(R) = \GL_2^+(R) \ltimes H(R).\nonumber
\end{equation}
The group law is given by standard matrix multiplication $g_1g_2$ for $g_1,g_2\in \GL_2^+(R)$ and by the rules
\begin{align}
	[(\lambda,\mu),\kappa]\cdot [(\lambda',\mu'),\kappa'] &= [(\lambda+\lambda',\mu+\mu'),\kappa+\kappa'+\lambda\mu'-\mu\lambda'] \text{ and }\nonumber \\
	[(\lambda,\mu),\kappa]\times g_1 &= g_1\times [\det(g_1)^{-1}(\lambda,\mu)g_1,\det(g_1)^{-1}\kappa].\nonumber
\end{align}

For $N\in \N$ we define the lattices
\begin{equation}
	\Gamma_0(N)^J = \{ (g,[(\lambda,\mu),\kappa])\in G^J(\Q)\colon g\in \Gamma_0(N),\, \lambda,\mu,\kappa \in \Z\}.\nonumber
\end{equation}
Note that $\Gamma_0(1)^J=\Gamma^J$ is the usual Jacobi group. All these congruence subgroups contain
\begin{equation}
	\Gamma_{\infty}^{J} = \left\{ \left(\left(\begin{matrix} 1 & n\\ 0 & 1\end{matrix}\right),[(0,\mu),\kappa]\right)\colon n,\mu,\kappa\in \Z\right\}. \nonumber
\end{equation}

Given a function $f\colon \Hb^{(1)}\times \C \to \C$ we define the usual slash action\footnote{This is only a proper group action when restricted to $\textrm{SL}_2(\R)\ltimes H(\R)$.} by
\begin{equation}
	[f\vert_{k,m}g](\tau,z) = (c\tau+d)^{-k}e\left(-ml\frac{cz^2}{c\tau+d}\right)f\left(\frac{a\tau+b}{c\tau+d},\frac{lz}{c\tau+d}\right) \text{ for }g=\left(\begin{matrix} a&b\\c&d\end{matrix}\right)\in \GL_2^+(\R) \nonumber
\end{equation}
and $l=ad-bc$. Furthermore,
\begin{equation}
	[f\vert_{k,m}[(\lambda,\mu),\kappa]](\tau,z) = e(m\lambda^2+2m\lambda z+\lambda\mu+\kappa)f(\tau,z+\lambda\tau+\mu).\nonumber
\end{equation}

A holomorphic function $f\colon \Hb\times \C\to\C$ is a Jacobi form of weight $k$ and index $m$ with respect to $\Gamma_0(N)^J$ if the following three conditions hold
\begin{enumerate}
	\item For all $\gamma\in \Gamma_0(N)$ we have $f\vert_{k,m}\gamma = f$;
	\item For all $\lambda,\mu\in \Z$ we have $f\vert_{k,m}[(\lambda,\mu),0] = f$;
	\item For any $M\in \SL_2(\Z)$ we have a Fourier expansion 
	\begin{equation}
		[f\vert_{k,m}M](\tau,z) = \sum_{\substack{n\in n_M^{-1}\Z, r\in \Z,\\ r^2\leq 4mn}}c_f^M(n,r)e(n\tau+rz).\nonumber
	\end{equation}
\end{enumerate}
We denote the space of all such functions by $J_{k,m}(\Gamma_0(N)^J)$. If $f\in J_{k,m}(\Gamma_0(N)^J)$ satisfies $c_f^M(n,r)=0$ unless $r^2<4mn$, then we call $f$ a cusp form. The space of all cusp forms is denoted by $J^{\textrm{cusp}}_{k,m}(\Gamma_0(N)^J)$.

To shorten notation we write $J_{k,m}=J_{k,m}(\Gamma_0(1)^J)$ and similarly $J^{\textrm{cusp}}_{k,m} = J^{\textrm{cusp}}_{k,m}(\Gamma_0(1)^J)$. Furthermore, for $M=I_2$ the identity matrix, we have $n_M=1$ and we write $c^M_f(n,r) = c_f(n,r)$. In this case we record the Fourier expansion
\begin{equation}
	f(\tau,z) = \sum_{\substack{n,r\in \Z, \\ r^2< 4mn}}c_f(n,r)e(n\tau+rz),\nonumber
\end{equation}
for $f\in J^{\textrm{cusp}}_{k,m}(\Gamma_0(N)^J)$.

The space $J^{\textrm{cusp}}_{k,m}(\Gamma_0(N)^J)$ is equipped with the Petersson inner product
\begin{equation}
	\langle f,g\rangle_{\Gamma_0(N)^J} = \int_{\Gamma_0(N)^J\backslash (\Hb^{(1)}\times \C)} f\overline{g}\mu_{k,m}^2 dV \nonumber
\end{equation}
where in the coordinates $\tau=u+iv$ and $z=x+iy$ we have $dV=v^{-3}dudvdxdy$ and $\mu_{k,m}(\tau,z) = v^{k/2}e^{-2\pi y^2/v}$.

Note that since $\Gamma_0(N)\subseteq \textrm{SL}_2(\Z)$ we have the inclusion
\begin{equation}
	J_{k,m}^{\textrm{cusp}} \subseteq J_{k,m}^{\textrm{cusp}}(\Gamma_0(N)^J).\nonumber
\end{equation}
This is not an isometric inclusion. Indeed, if $f\in J_{k,m}^{\textrm{cusp}}$, then we have
\begin{equation}
	\langle f,f\rangle_{\Gamma_0(N)^J} = [\Gamma^J\colon\Gamma_0(N)^J] \langle f,f\rangle_{\Gamma^J}.\label{eq:scaling_of_norm}
\end{equation}
We have $[\Gamma^J\colon\Gamma_0(N)^J]=[\SL_2(\R)\colon\Gamma_0(N)]$.

Finally, let us record the following version of the Petersson formula for Jacobi forms.

\begin{theorem}[Petersson formula for Jacobi forms]\label{th:petersson}
Let $k\geq 4$ be even, $m,n,r\in \Z$ with $D=r^2-4mn<0$ and $N\in \N$. Put 
\begin{equation}
	\lambda_{k,m}(D) =\frac{\Gamma(k-3/2)}{4\pi^{k-3/2}}m^{k-2}\vert D\vert^{3/2-k}.\nonumber
\end{equation}
Further, let $\mathcal{O}_{k,m}(N)$ be an orthogonal basis of $S_{k,m}^{\textrm{cusp}}(\Gamma_0(N)^J)$. Then we have
\begin{multline}
	\lambda_{k,m}(D)\sum_{f\in \mathcal{O}_{k,m}(N)} \frac{\vert c_f(n,r)\vert^2}{\langle f,f\rangle_{\Gamma_0(N)^J}} \\ = 1 + \frac{i^k\pi}{\sqrt{2m}}\sum_{\pm}\sum_{\substack{c\geq 1\\ N\mid c}}c^{-\frac{3}{2}}H_{m,c}^{\pm}(n,r)\cdot e\left(\pm\frac{r^2}{2mc}\right)J_{k-3/2}\left(\frac{\pi\vert D\vert}{mc}\right),\nonumber
\end{multline}	
for the Kloosterman type sums $H_{m,c}^{\pm}(n,r)$ defined in \eqref{eq:def_KS}.
\end{theorem}
\begin{proof}
The proof is standard and uses basic properties of Poincar\'e series. See \cite[(1) and (2)]{GKZ} as well as \cite[(2)]{Ko}. The adaptions to be made to cover $\Gamma_0(N)^J$ are straight forward. See also \cite{Ho}.
\end{proof}

\section{The method of Iwaniec} \label{subsec}

Our goal in this section is to establish Theorem~\ref{th:bound_jacobi}. We turn towards the technical main part of the paper, which is concerned with estimating the geometric side of the Petersson formula in order to deduce strong bounds for Fourier coefficients of Jacobi forms. We will do so following the method of Iwaniec from \cite{Iw_half}. See also \cite[Section~5.3]{Iw}. 

Throughout this section we let $f\in S_{k,m}^{\textrm{cusp}}$ be a Jacobi form of weight $k$ and index $m$. We will take $r,m,n\in \N$ with $D=r^2-4mn<0$ a fundamental discriminant and assume that $m\ll \vert D\vert^{\frac{1}{2}}$. In this section all implicit constants are allowed to depend on $\epsilon>0$, which we imagine to be tiny and on the weight $k\geq 4$. We will simply write $\ll$ instead of $\ll_{k,\epsilon}$ from now on.

\subsection{Setting up the scene}

For the set-up we fix a parameter $P\ll \vert D\vert$, which will be optimized at the end. We start with the simple observation
\begin{align}
	&\lambda_{k,m}(D)\sum_{\substack{P\leq p\leq 2P\\ p\nmid mD}}\log(p)\frac{\vert c_{f_{\circ}}(n,r)\vert^2}{\langle f_{\circ},f_{\circ}\rangle_{\Gamma_0(p)^J}} \leq \lambda_{k,m}(D)\sum_{\substack{P\leq p\leq 2P\\ p\nmid mD}} \log(p)\sum_{f\in \mathcal{O}_{k,m}(p)}\frac{\vert c_{f}(n,r)\vert^2}{\langle f,f\rangle_{\Gamma_0(p)^J}} \nonumber \\
	&\qquad = \sum_{\substack{P\leq p\leq 2P\\ p\nmid mD}}\log(p) + \frac{i^k\pi}{\sqrt{2m}} \sum_{\pm}\sum_{\substack{P\leq p\leq 2P\\ p\nmid mD}}\log(p) \sum_{\substack{c\geq 1\\ p\mid c}}\frac{H_{m,c}^{\pm}(n,r)}{c^{\frac{3}{2}}}e\left(\pm\frac{r^2}{2mc}\right)J_{k-3/2}\left(\frac{\pi\vert D\vert}{mc}\right). \nonumber
\end{align}
Obviously, by the prime number theorem, we have
\begin{equation}
	\sum_{\substack{P\leq p\leq 2P\\ p\nmid mD}}\log(p) \ll P.\nonumber
\end{equation}
On the other hand, by \eqref{eq:scaling_of_norm} we find that
 \begin{align}
 	\lambda_{k,m}(D)\sum_{\substack{P\leq p\leq 2P\\ p\nmid mD}}\log(p)\frac{\vert c_{f_{\circ}}(n,r)\vert^2}{\langle f_{\circ},f_{\circ}\rangle_{\Gamma_0(p)^J}} &= \lambda_{k,m}(D)\frac{\vert c_{f_{\circ}}(n,r)\vert^2}{\langle f_{\circ},f_{\circ}\rangle_{\Gamma^J}}\sum_{\substack{P\leq p\leq 2P\\ p\nmid mD}}\frac{\log(p)}{p+1} \nonumber\\
 	&\gg \vert D\vert^{-\epsilon}\lambda_{k,m}(D)\frac{\vert c_{f_{\circ}}(n,r)\vert^2}{\langle f_{\circ},f_{\circ}\rangle_{\Gamma^J}}. \nonumber
 \end{align}
Thus, we end up with
\begin{equation}
	\left(\frac{m}{\vert D\vert}\right)^{k-\frac{3}{2}}\frac{\vert c_{f_{\circ}}(n,r)\vert^2}{\langle f_{\circ},f_{\circ}\rangle_{\Gamma^J}} \ll \vert D\vert^{\epsilon}(m^{\frac{1}{2}}P+\vert \mathcal{S}\vert),\label{eq:starting_point}
\end{equation}
for
\begin{equation}
	\mathcal{S} = \sum_{\pm }\sum_{c=1}^{\infty} \omega(c)\frac{H_{m,c}^{\pm}(n,r)}{c^{\frac{3}{2}}}e\left(\pm\frac{r^2}{2mc}\right)J_{k-3/2}\left(\frac{\pi\vert D\vert}{mc}\right) \nonumber
\end{equation}
with
\begin{equation}
	\omega(c) = \sum_{\substack{P\leq p\leq 2P\\ p\nmid mD \\ p\mid c}}\log(p).\nonumber
\end{equation}

The result will follow from a careful analysis of $\mathcal{S}$, which we conduct in the following subsections.

\subsection{The tails}

We fix further parameters $0<C<K\ll \vert D\vert^2$ to be optimized later. We use these to split the sum
\begin{equation}
	\mathcal{S} = \mathcal{S}^{\flat} + \mathcal{S}^{\ast}+\mathcal{S}^{\sharp}\nonumber
\end{equation}
according to $c\leq C$, $C<c<K$ and $K\leq c$.

\begin{lemmy}\label{lm:tails}
For $C\leq \frac{\vert D\vert}{m}\leq K$ and $k\geq 4$ we have
\begin{equation}
	\mathcal{S}^{\flat}\ll \left(\frac{\vert D\vert}{m}\right)^{-\frac{1}{2}}\vert D\vert^{\epsilon}C \nonumber
\end{equation}
and 
\begin{equation}
	\mathcal{S}^{\sharp} \ll \left(\frac{\vert D\vert}{m}\right)^{\frac{3}{2}}\vert D\vert^{\epsilon}K^{-1}. \nonumber
\end{equation}
\end{lemmy}
\begin{proof}
We start by recalling the bound
\begin{equation}
	J_{k-\frac{3}{2}}(x) \ll \min(x^{-\frac{1}{2}}, x^{k-\frac{3}{2}}).\label{eq:bessel}
\end{equation}
With this and Lemma~\ref{lm:imp_Kl_bound} we can estimate
\begin{align}
	\mathcal{S}^{\flat} &=\sum_{\pm}\sum_{1\leq c\leq C} \frac{\omega(c)}{c^{\frac{3}{2}}}H_{m,c}^{\pm}(n,r) e\left(\pm\frac{r^2}{2mc}\right)J_{k-3/2}\left(\frac{\pi\vert D\vert}{mc}\right) \nonumber\\
	&\ll \left(\frac{\vert D\vert}{m}\right)^{-\frac{1}{2}}\sum_{c=1}^Cc^{\epsilon}\cdot (c,D)^{\frac{1}{2}}\omega(c) \ll \left(\frac{\vert D\vert}{m}\right)^{-\frac{1}{2}}\vert D\vert^{\epsilon}C.\nonumber
\end{align}
The second estimate is similar.
\end{proof}

\begin{rem}
From here we can recover the estimate from Theorem~\ref{Kohnen}. Indeed, taking $P$ constant and $K=C=\frac{\vert D\vert}{m}$, yields
\begin{equation}
	\left(\frac{m}{\vert D\vert}\right)^{k-\frac{3}{2}}\frac{\vert c_{f_{\circ}}(n,r)\vert^2}{\langle f_{\circ},f_{\circ}\rangle_{\Gamma^J}} \ll \sqrt{m}+\frac{\vert D\vert^{\frac{1}{2}+\epsilon}}{m^{\frac{1}{2}}}.\nonumber
\end{equation}
This translates to the desired bound.
\end{rem}

We continue our treatment of $\mathcal{S}^{\ast}$. First, we write $c=qt$ for $t\mid(2mD)^{\infty}$ and $(t,q)=1$. Note that by construction $w(c)=w(q)$. Thus we obtain
\begin{equation}
	\mathcal{S}^{\ast} = \sum_{t\mid (2mD)^{\infty}}\underbrace{\sum_{\pm}\sum_{\substack{C<qt<D\\ (q,2mD)=1}}\frac{\omega(q)}{(qt)^{\frac{3}{2}}}H_{m,qt}^{\pm}(n,r)e\left(\pm\frac{r^2}{2mc}\right)J_{k-3/2}\left(\frac{4\pi \vert D\vert}{mtq}\right)}_{=\mathcal{T}_t}. \nonumber
\end{equation}
Using arguments as before we can trivially bound
\begin{equation}
	\mathcal{T}_t \ll t^{\epsilon-1}(t,D)^{\frac{1}{2}}\vert D\vert^{\epsilon}\left(\frac{\vert D\vert}{m}\right)^{\frac{1}{2}}.\nonumber
\end{equation}
Using this we can estimate
\begin{equation}
	\sum_{\substack{ t\mid (2mD)^{\infty}\\ t\geq T}} \mathcal{T}_t \ll \vert D\vert^{\epsilon}\left(\frac{\vert D\vert}{m}\right)^{\frac{1}{2}}\sum_{\substack{ t\mid (2mD)^{\infty}\\ t\geq T}}t^{\epsilon-1}(t,D)^{\frac{1}{2}} \ll \left(\frac{\vert D\vert}{m}\right)^{\frac{1}{2}}\vert D\vert^{\epsilon}T^{-1}.\nonumber
\end{equation}
On the other hand we can factor the Kloosterman sum using \eqref{eq:factorization}. Indeed, we obtain
\begin{equation}
	H_{m,qt}^{\pm}(n,r) = H_{mq,t}^{\pm}(n\overline{q},r)H_{mt,q}^{\pm}(n\overline{t},r). \nonumber
\end{equation}
The upshot is that $H_{mq,t}^{\pm}(n\overline{q},r)$ only depends on the congruence class of $q$ modulo $t$. Furthermore, if $D$ is a fundamental discriminant, we have the bound
\begin{equation}
	H_{mq,t}^{\pm}(n\overline{q},r) \ll t^{1+\epsilon}(t,D)^{\frac{1}{2}}\nonumber
\end{equation}
from Lemma~\ref{lm:imp_Kl_bound}. With these observations made we can write
\begin{equation}
	\mathcal{T}_t \ll \max_{\pm} {\sum_{s\text{ mod }4t}}^{\ast} t^{\epsilon-\frac{1}{2}}(t,D)^{\frac{1}{2}}\left\vert \sum_{\substack{C<qt<K,\\ q\equiv s \text{ mod }4t,\\ (q,2mD)=1}}\frac{w(q)}{q^{\frac{3}{2}}}H_{mt,q}^{\pm}(n\overline{t},r)e\left(\pm\frac{r^2}{2mtq}\right)J_{k-3/2}\left(\frac{4\pi \vert D\vert}{mtq}\right)\right\vert.
\end{equation}
After putting everything together we arrive at the following result.

\begin{lemmy}\label{lm:after_red}
For large parameters $P,T$ and for $0<C\leq \frac{\vert D\vert}{m} \leq D$ we have
\begin{multline}
	\left(\frac{m}{\vert D\vert}\right)^{k-\frac{3}{2}}\frac{\vert c_{f_{\circ}}(n,r)\vert^2}{\langle f_{\circ},f_{\circ}\rangle_{\Gamma^J}} \ll \vert D\vert^{\epsilon}m^{\frac{1}{2}}P + \vert D\vert^{\epsilon}\left(\frac{\vert D\vert}{m}\right)^{-\frac{1}{2}}C + \vert D\vert^{\epsilon}\left(\frac{\vert D\vert}{m}\right)^{\frac{3}{2}}K^{-1} \\
	+ \vert D\vert^{\epsilon}\left(\frac{\vert D\vert}{m}\right)^{\frac{1}{2}}T^{\epsilon-1}  + \max_{\pm}\sum_{\substack{t\mid(2mD)^{\infty}\\ t<T}}t^{\epsilon-\frac{1}{2}}(t,D)^{\frac{1}{2}} {\sum_{s\text{\normalfont{ mod }}t}}^{\ast}\vert\mathcal{T}_{t,s}^{\pm}\vert, \nonumber
\end{multline}
for
\begin{equation}
	\mathcal{T}_{t,s}^{\pm} = \sum_{\substack{C<qt<K,\\ q\equiv s \text{\normalfont{ mod }}4t,\\ (q,2mD)=1}}\frac{w(q)}{q^{\frac{3}{2}}}H_{mt,q}^{\pm}(n\overline{t},r)e\left(\pm\frac{r^2}{2mtq}\right)J_{k-3/2}\left(\frac{4m\vert D\vert}{mtq}\right).\nonumber
\end{equation}
\end{lemmy}

\subsection{Estimation of $\mathcal{T}_{t,s}^{\pm}$}

This section is dedicated to the estimation of the sums $\mathcal{T}_{t,s}^{\pm}$. We start by inserting the evaluation of $H_{mt,c}^{\pm}(n\overline{t},r)$ from Lemma~\ref{lm:eval_KS} into the definition of $\mathcal{T}_{r,s}^{\pm}$. Note that $\epsilon_q=\epsilon_s$ since $q \equiv s\text{ mod }4.$ We arrive at
\begin{multline}
	\mathcal{T}_{t,s}^{\pm} = \epsilon_s^2\sum_{\substack{C<qt<K,\\ q\equiv s \text{ mod }4t,\\ (q,2mD)=1}}\frac{w(q)}{q^{\frac{1}{2}}}\left(\frac{-D}{q}\right) \\
	\cdot \sum_{v^2\equiv 1\text{ mod } q}e\left(\frac{\overline{2mt}Dv}{q}\mp \frac{\overline{2mt}r^2}{q}\pm \frac{r^2}{2mtq}\right) J_{k-3/2}\left(\frac{4\pi \vert D\vert}{mtq}\right).\nonumber
\end{multline}
We start by massaging this expression a little bit. Using the reciprocity formula $\frac{\overline{y}}{x}+\frac{\overline{x}}{y} \equiv \frac{1}{xy}\text{ mod }1$ we can simplify
\begin{equation}
	\mp \frac{\overline{2mt}r^2}{q}\pm \frac{r^2}{2mtq} \equiv \mp \frac{\overline{q}r}{2mt} \text{ mod }1. \nonumber
\end{equation}

 We first observe that we can replace the sum over $v^2\equiv 1\text{ mod }q$ by summing over $ab=q$ with $(a,q)=1$ and setting $v=a\overline{a}-b\overline{b}$. Thus we have
\begin{equation}
	\sum_{v^2\equiv 1\text{ mod } q}e\left(\frac{\overline{2mt}Dv}{q}\pm \frac{\overline{q}r^2}{2mt}\right) = \sum_{\substack{ab=q \\ (a,b)=1}}e\left(\frac{\overline{2mt}D\overline{a}}{b}-\frac{\overline{2mt}D\overline{b}}{a}\pm \frac{\overline{ab}r^2}{2mt}\right)
\end{equation}
If $a<b$ we use the reciprocity formula $\frac{\overline{x}}{y}+\frac{\overline{y}}{x} \equiv \frac{1}{xy} \text{ mod }1$ to replace $\frac{\overline{2mta}}{b}$ by $-\frac{\overline{b}}{2mta}+\frac{1}{2mtab}$. On the other hand, if $a<b$ we do the same for $\frac{\overline{2mtb}}{a}$.  After doing so, we can swap the roles of $a$ and $b$, so that we obtain
\begin{equation}
	\sum_{v^2\equiv 1\text{ mod } q}e\left(\frac{\overline{2mt}Dv}{q}\pm \frac{\overline{q}r^2}{2mt}\right) = \sum_{\epsilon_2 = \pm 1}\sum_{\substack{ab=q \\ (a,b)=1 \\ a<b}}e\left(\epsilon_2\frac{D\overline{b}}{2mta}+\epsilon_2\frac{\overline{2mt}D\overline{b}}{a}-\epsilon_2\frac{D}{2mtab}\pm \frac{\overline{ab}r^2}{2mt}\right). \nonumber
\end{equation}
Next, define
\begin{equation}
	j_{\epsilon_2}(x) = x^{\frac{1}{2}}e(\epsilon_2 x) J_{k-3/2}(8\pi x).\nonumber
\end{equation}
We arrive at
\begin{multline}
	\mathcal{T}_{t,s}^{\epsilon_1} = \epsilon_s^2\left(\frac{-1}{s}\right)\left(\frac{2mt}{\vert D\vert}\right)^{\frac{1}{2}}\sum_{\epsilon_2=\pm 1}\sum_{\substack{C<abt<K,\\ ab\equiv s \text{ mod }4t,\\ (ab,2mD)=1 \\ a<b, \, (a,b)=1}}w(ab)\left(\frac{D}{ab}\right)\\ \cdot e\left(\epsilon_2\frac{D\overline{b}}{2mta}+\epsilon_2\frac{\overline{2mt}D\overline{b}}{a}+\epsilon_1 \frac{\overline{ab}r^2}{2mt}\right) j_{\epsilon_2}\left(\frac{ \vert D\vert}{2mtab}\right), \nonumber
\end{multline}
for $\epsilon_1=\pm 1$.

We temporarily define
\begin{equation}
	S_a(B) = \sum_{\substack{\max(a,\frac{C}{at})<b<B,\\ ab\equiv s\text{ mod }4t}} \omega(ab) \left(\frac{D}{b}\right) e\left(f(a,t)\frac{\overline{b}}{2mta}\right).\nonumber
\end{equation}
for 
\begin{equation}
	f(a,t) = \epsilon_2 D(1+2mt\cdot \overline{2mt})+\epsilon_1a\cdot \overline{a}r^2. \label{def_f}
\end{equation}
Here $\overline{2mt}$ is a (representative of) the modular inverse of $2mt$ modulo $a$. Similarly $\overline{a}$ is an inverse of $a$ modulo $2mt$. We drop the dependence of $f$ on $r,m,D$ as well as on the signs $\epsilon_1,\epsilon_2$ from the notation. We compute 
\begin{multline}
	\mathcal{T}_{t,s}^{\epsilon_1} \ll \left(\frac{2mt}{\vert D\vert }\right)^{\frac{1}{2}} \max_{\epsilon_1=\pm} \sum_{\substack{a<\sqrt{K/t},\\ (a,2mD)=1}}  \bigg\vert\sum_{\substack{\max(a,\frac{C}{at})<b<K/at,\\ ab\equiv s \text{ mod }4t,\\ (b,2mDa)=1}}w(ab)\left(\frac{D}{b}\right) \\ \cdot e\left(\epsilon_2\frac{D\overline{b}}{2mta}+\epsilon_2\frac{\overline{2mt}D\overline{b}}{a}+\epsilon_1 \frac{\overline{ab}r^2}{2mt}\right) j_{\epsilon_2}\left(\frac{ \vert D\vert}{2mtab}\right) \bigg\vert \nonumber 
\end{multline}
After applying partial integration in the $b$-sum we arrive at
\begin{equation}
	\mathcal{T}_{t,s}^{\epsilon_1} \ll C^{-1}\left(\frac{\vert D\vert t}{m}\right)^{\frac{1}{2}} \max_{\epsilon_1=\pm} \sum_{\substack{a<\sqrt{K/t},\\ (a,2mD)=1}} \max_{B\leq K/at} \vert S_a(B)\vert, \label{eq:Tst}
\end{equation}
here we have only used the easy bounds $j(x) \ll x^{\frac{1}{2}}$, $j(x)\ll 1$ and $j'(x)\ll 1$.

Note that on the support of the $a$- and $b$-sums we have $\omega(ab) = \omega(a)+\omega(b)$ we can thus decompose
\begin{equation}
	S_a(B) = V_a+\omega(a)V_a'\nonumber
\end{equation}
with
\begin{equation}
	V_a = \sum_{\substack{\max(a,\frac{C}{at})<b<B,\\ ab\equiv s\text{ mod }4t}} \omega(b) \left(\frac{D}{b}\right) e\left(f(a,t)\frac{\overline{b}}{2mta}\right)\nonumber
\end{equation}
and 
\begin{equation}
	V_a' = \sum_{\substack{\max(a,\frac{C}{at})<b<B,\\ ab\equiv s\text{ mod }4t}} \left(\frac{D}{b}\right) e\left(f(a,t)\frac{\overline{b}}{2mta}\right).\nonumber
\end{equation}
In the rest of the section we will estimate $V_a$. The same bound can be obtained for $\omega(a)V_a'$ using similar arguments.

\begin{lemmy}\label{lm:V_a}
For $(KT)^{\frac{1}{2}} \ll \vert D\vert$ we have
\begin{equation}
		V_a \ll amP\vert D\vert^{\frac{1}{2}+\epsilon}\nonumber
\end{equation}
and
\begin{equation}
	V_a \ll \vert D\vert^{\epsilon}\left[\frac{K}{P^{\frac{1}{2}}ta}+\frac{K}{(ta)^{\frac{3}{2}}} + \frac{K^{\frac{1}{2}}m^{\frac{1}{4}}}{a^{\frac{1}{4}}t^{\frac{1}{2}}}(1+\frac{P}{t})^{\frac{1}{2}}\right]. \nonumber
\end{equation}
The same bounds hold for $\omega(a)V_a'$.
\end{lemmy}
\begin{proof}
Put $A=\max(a,\frac{C}{at})$. We start with an estimate that will be useful when $a$ is relatively small. In this case we open the definition of $\omega(b)$ and obtain
\begin{equation}
	V_a = \sum_{\substack{P<p<2P,\\ p\nmid 2mDa}} \log(p)\left(\frac{D}{p}\right)\sum_{\substack{A/p < l <B/p,\\ (l,2mDa)=1, \\ apl\equiv s\text{ mod }4t}}\left(\frac{D}{l}\right)e\left(f(a,t)\frac{\overline{pl}}{2mta}\right).\nonumber
\end{equation}
Using the Polya-Vinogradov inequality we estimate
\begin{equation}
	V_a \ll \log(P)\sum_{\substack{P<p<2P,\\ p\nmid 2mDa}} {\sum_{\substack{\lambda\text{ mod }4tma, \\ \lambda \equiv \overline{ap}s \text{ mod }4t}}}^{\ast}\left\vert \sum_{\substack{A/p < l <B/p,\\  l\equiv \lambda \text{ mod }4tma}}\left(\frac{D}{l}\right)\right\vert.\nonumber
\end{equation}
Note that, since $D$ is a fundamental discriminant and $tma\ll \sqrt{KT} \ll \vert D\vert$, the character $\left(\frac{D}{l}\right)$ is non-principal. Thus, applying the Polya-Vinogradov inequality we can estimate the $l$-sum by $\vert D\vert^{\frac{1}{2}}\log(\vert D\vert)$. We obtain the bound
\begin{equation}
	V_a \ll am P\vert D\vert^{\frac{1}{2}+\epsilon}.\label{eq:small}
\end{equation}
This is the first bound.

The second bound, which is good for $a$ not too small, is obtained by applying Cauchy-Schwarz to find
\begin{equation}
	V_a^2\ll \frac{B}{P}\sum_{\substack{\frac{A}{P}<l<\frac{B}{P},\\(l,2ma)=1}}\left\vert \sum_{\substack{P_1<p<P_2\\ p\nmid 2mDa,\\ alp\equiv s \text{ mod }4t}}\log(p)\left(\frac{D}{p}\right)e\left(f(a,t)\frac{\overline{pl}}{2mta}\right) \right\vert^2,\nonumber
\end{equation}
where $P_1=\max(P,A/l)$ and $P_2=\min(2P,B/l)$. As usual we open the square and take the $l$-sum inside. We obtain
\begin{equation}
	V_a^2 \ll \frac{B}{P}\log(P)^2 \sum_{\substack{P<p\leq p'<2P,\\\ p\equiv p'\text{ mod }4t \\ p,p'\nmid 2mDa}} \left\vert \sum_{\substack{\frac{A}{p'}<l<\frac{B}{p} \\ apl\equiv s \text{ mod }4t\\(l,2ma)=1}} e\left(f(a,t)\cdot \frac{p-p'}{4t}\cdot \frac{2\overline{pp'l}}{ma}\right)\right\vert.\nonumber 
\end{equation}
Estimating the $l$-sum using Lemma~\ref{lm:completeion} yields 
\begin{equation}
	V_a^2 \ll \vert D\vert^{\epsilon}\frac{B}{P}\sum_{\substack{P<p\leq p'<2P,\\\ p\equiv p'\text{ mod }4t \\ p,p'\nmid 2mDa}} \bigg[\frac{B}{pma}\Big(f(a,t)\cdot \frac{p-p'}{4t},ma\Big)  +(ma)^{\frac{1}{2}}\bigg].\label{eq:above}
\end{equation} 
At this point we have to investigate the greatest common divisor that appears in this bound. To do so we recall the definition of $f(a,t)$ from \eqref{def_f} and note that $(f(a,t),a)=1$. This is because $a$ is co-prime to $2D$. Thus, we obtain
\begin{equation}
	\Big(f(a,t)\cdot \frac{p-p'}{4t},ma\Big) = \Big(\frac{p-p'}{4t},a\Big)\cdot \Big(f(a,t)\cdot \frac{p-p'}{4t},m\Big)
\end{equation}
Here we compute that $p\mid m$ and $p\mid f(a,t)$ precisely when $p\mid (\epsilon_1+\epsilon_2)r^2$. Therefore, the best we can do is to estimate
\begin{equation}
	\Big(f(a,t)\cdot \frac{p-p'}{4t},ma\Big) \leq m\cdot \Big(\frac{p-p'}{4t},a\Big).\nonumber
\end{equation}
Plugging this in \eqref{eq:above} gives
\begin{align}
	V_a^2 &\ll \vert D\vert^{\epsilon}\frac{B}{P}\sum_{d\mid a}\sum_{\substack{P<p<2P,\\\  p\nmid 2mDa}} \sum_{\substack{P< p' <2P,\\ p'\equiv p\text{ mod }4td}} \bigg[\frac{B}{pa} d+(ma)^{\frac{1}{2}}\bigg]\nonumber \\
	&\ll \vert D\vert^{\epsilon}\left[\frac{B^2}{P} + \frac{B^2}{ta} + B(ma)^{\frac{1}{2}}(1+\frac{P}{t})\right].\label{eq:large} 
\end{align} 
where the first contribution arises from the diagonal. Inserting the bound $B\leq \frac{K}{at}$ gives the desired second estimate.
\end{proof}

\begin{lemmy}\label{Tst}
Let $P,C,K,T>0$ be parameters satisfying $P\ll \vert D\vert$, $0<C\leq \frac{\vert D\vert}{m} \leq K\ll \vert D\vert^2$ and $\sqrt{KT}\ll \vert D\vert$.  We have
\begin{equation}
	\mathcal{T}_{t,s}^{\epsilon_1} \ll \frac{\vert D\vert^{\frac{1}{2}+\epsilon}}{m^{\frac{1}{2}}C}\left(\frac{\vert D\vert^{\frac{1}{10}}m^{\frac{1}{5}}P^{\frac{1}{5}}K^{\frac{4}{5}}}{t^{\frac{7}{10}}}+\frac{K}{P^{\frac{1}{2}}t^{\frac{1}{2}}}+\frac{K^{\frac{7}{8}}m^{\frac{1}{4}}}{t^{3/8}}(1+P/t)^{\frac{1}{2}}\right).\nonumber
\end{equation}
\end{lemmy}
\begin{proof}
Let $0<X\leq \sqrt{\frac{K}{t}}$ be a parameter to be chosen shortly. We insert the results from Lemma~\ref{lm:V_a} into \eqref{eq:Tst} and execute the $a$-sum. More precisely, we use the first bound of Lemma~\ref{lm:V_a} for $a\leq X$ and the second one for $a>X$. We get
\begin{equation}
	\mathcal{T}_{t,s}^{\epsilon_1} \ll \frac{\vert D\vert^{\frac{1}{2}+\epsilon}}{m^{\frac{1}{2}}C}\left(X^2mP\vert D\vert^{\frac{1}{2}}t^{\frac{1}{2}}+\frac{K}{tX^{\frac{1}{2}}}+ \frac{K}{P^{\frac{1}{2}}t^{\frac{1}{2}}}+\frac{K^{\frac{7}{8}}m^{\frac{1}{4}}}{t^{3/8}}(1+P/t)^{\frac{1}{2}}\right).\nonumber
\end{equation}
We choose $X = \frac{K^{\frac{2}{5}}}{\vert D\vert^{\frac{1}{5}}(mP)^{\frac{2}{5}}t^{\frac{3}{5}}}$ to obtain the desired result.
\end{proof}

\subsection{The endgame}

We now insert the bound from Lemma~\ref{Tst} into Lemma~\ref{lm:after_red}. Estimating the $s$-sum trivially and executing the remaining $t$-sum using the Rankin-trick we obtain
\begin{multline}
		\left(\frac{m}{\vert D\vert}\right)^{k-\frac{3}{2}}\frac{\vert c_{f_{\circ}}(n,r)\vert^2}{\langle f_{\circ},f_{\circ}\rangle_{\Gamma^J}} \\ \ll \vert D\vert^{\epsilon} m^{\frac{1}{2}}P + \vert D\vert^{\epsilon}\left(\frac{\vert D\vert}{m}\right)^{-\frac{1}{2}}C + \vert D\vert^{\epsilon}\left(\frac{\vert D\vert}{m}\right)^{\frac{3}{2}}K^{-1}+\vert D\vert^{\epsilon}\left(\frac{\vert D\vert}{m}\right)^{\frac{1}{2}}T^{-1} \\
		+ \frac{\vert D\vert^{\frac{1}{2}+\epsilon}}{m^{\frac{1}{2}}}\left(\frac{\vert D\vert^{\frac{1}{10}}m^{\frac{1}{5}}P^{\frac{1}{5}}T^{\frac{3}{10}}K^{\frac{4}{5}}}{C}+\frac{KT^{\frac{1}{2}}}{CP^{\frac{1}{2}}}+\frac{K^{\frac{7}{8}}m^{\frac{1}{4}}T^{\frac{1}{8}}(T^{\frac{1}{2}}+P^{\frac{1}{2}})}{C}\right). \nonumber
\end{multline}
Here we have estimated the $s$-sum trivially and the $t$-sum using the Rankin-trick.

We will now make a suitable choice for the parameters $P,C,K,T$. Note that we did not attempt to fully optimize this choice and we expect that different choices might lead to improved bounds in certain ranges. We first, for $0<\delta<\frac{1}{2}$, put $$C=(\vert D\vert/m)^{1-\delta},\, K=(\vert D  \vert/m)^{1+\delta} \text{ and } T=(\vert D\vert/m)^{\delta}.$$ Assuming $P\geq T$ and dropping some irrelevant terms leads to 
\begin{multline}
	D^{-\epsilon}\left(\frac{m}{\vert D\vert}\right)^{k-1}\frac{\vert c_{f_{\circ}}(n,r)\vert^2}{\langle f_{\circ},f_{\circ}\rangle_{\Gamma^J}}  \ll  \left(\frac{\vert D\vert}{m}\right)^{-\delta}+  m^{\frac{3}{10}}P^{\frac{1}{5}}\left(\frac{\vert D\vert}{m}\right)^{\frac{21}{10}\delta-\frac{1}{10}} \\+\frac{1}{P^{\frac{1}{2}}}\left(\frac{\vert D\vert}{m}\right)^{\frac{5}{2}\delta} +m^{\frac{1}{4}}P^{\frac{1}{2}}\left(\frac{\vert D\vert}{m}\right)^{2\delta-\frac{1}{8}}. \label{before_P}
\end{multline}
We equalize the second and the third term by choosing
\begin{equation}
	P=m^{-\frac{3}{10}}\left(\frac{\vert D\vert}{m}\right)^{\frac{1}{10}+\frac{2}{5}\delta}. \nonumber
\end{equation}
In this case $P\geq T$ holds as long as $m\leq \vert D\vert^{\frac{1+15\delta}{4+15\delta}}$. We arrive at
\begin{equation}
	D^{-\epsilon}\left(\frac{m}{\vert D\vert}\right)^{k-1}\frac{\vert c_{f_{\circ}}(n,r)\vert^2}{\langle f_{\circ},f_{\circ}\rangle_{\Gamma^J}}  \ll  \left(\frac{\vert D\vert}{m}\right)^{-\delta} +m^{\frac{3}{14}}\left(\frac{\vert D\vert}{m}\right)^{\frac{27}{14}\delta-\frac{1}{7}} +m^{\frac{1}{28}}\left(\frac{\vert D\vert}{m}\right)^{\frac{16}{7}\delta-\frac{3}{56}}. \nonumber
\end{equation}
We now take $m=\vert D\vert^{\sigma}$ for $0\leq \sigma\leq \frac{1}{2}$. We now choose 
\begin{equation}
	\delta=\frac{1}{72}\cdot\frac{3-5\sigma}{1-\sigma}. \nonumber
\end{equation}
This is positive and one verifies $P\geq T$ as long as $\sigma\leq \frac{39}{121}$. We get 
\begin{equation}
	D^{-\epsilon}\left(\frac{m}{\vert D\vert}\right)^{k-1}\frac{\vert c_{f_{\circ}}(n,r)\vert^2}{\langle f_{\circ},f_{\circ}\rangle_{\Gamma^J}}  \ll m^{\frac{5}{72}}\vert D\vert^{-\frac{3}{72}} +m^{\frac{25}{112}}\vert D\vert^{-\frac{1}{16}}. \nonumber
\end{equation}
We obtain that
\begin{equation}
	D^{-\epsilon}\left(\frac{m}{\vert D\vert}\right)^{k-1}\frac{\vert c_{f_{\circ}}(n,r)\vert^2}{\langle f_{\circ},f_{\circ}\rangle_{\Gamma^J}}  \ll \begin{cases}
		m^{\frac{5}{72}}\vert D\vert^{-\frac{3}{72}} &\text{ if }1\leq m\leq \vert D\vert^{\frac{21}{155}}, \text{ and }\\
		m^{\frac{25}{112}}\vert D\vert^{-\frac{1}{16}} &\text{ if } \vert D\vert^{\frac{21}{155}}\leq m\leq \vert D\vert^{\frac{7}{25}}.
	\end{cases} \nonumber
\end{equation}
Outside this range our bound is worse than Kohnen's result. The statement of Theorem~\ref{th:bound_jacobi} follows by interpolating the two cases above.


\end{document}